\documentclass{amsart}
\usepackage{amssymb,amsthm, amsmath, amsfonts}
\usepackage{graphics}
\usepackage{hyperref}
\usepackage[all]{xy}
\usepackage{enumerate}
\usepackage[mathscr]{eucal}
\usepackage{bbm}

\theoremstyle{plain}
\newtheorem{theorem}{Theorem}[section]
\newtheorem{lemma}[theorem]{Lemma}
\newtheorem{proposition}[theorem]{Proposition}

\newtheorem{corollary}[theorem]{Corollary}
\numberwithin{equation}{section}

\theoremstyle{definition}

\newtheorem{definition}[theorem]{Definition}

\newtheorem{remark}[theorem]{Remark}

\DeclareMathOperator{\Mod}{-Mod}
\DeclareMathOperator{\lfMod}{-lfMod}
\DeclareMathOperator{\Hom}{Hom}
\DeclareMathOperator{\Tor}{Tor}
\DeclareMathOperator{\hd}{hd}
\DeclareMathOperator{\gd}{gd}

\DeclareMathOperator{\td}{td}

\DeclareMathOperator{\coker}{coker}
\DeclareMathOperator{\Image}{Im}
\DeclareMathOperator{\kernel}{ker}
\DeclareMathOperator{\reg}{reg}

\newcommand{\mi}{\mathfrak{m}}

\newcommand{\C}{{\mathscr{C}}}
\newcommand{\tC}{{\underline{\mathscr{C}}}}

\newcommand{\FI}{{\mathscr{FI}}}

\title{Two homological proofs of the Noetherianity of $FI_G$}
\author{Liping Li}
\address{College of Mathematics and Computer Science; Performance Computing and Stochastic Information Processing (Ministry of Education), Hunan Normal University, Changsha, Hunan 410081, China.}
\email{lipingli@hunnu.edu.cn.}
\thanks{The author is supported by the National Natural Science Foundation of China 11541002, the Construct Program of the Key Discipline in Hunan Province, and the Start-Up Funds of Hunan Normal University 830122-0037. He also thanks Wee Liang Gan for the proposal to find a homological proof of the Noetherianity of $\FI$; and thanks Eric Ramos for notifying the author of his results and kindly sending the author a preprint.}

\begin{document}

\begin{abstract}
We give two homological proofs of the Noetherianity of the category $FI$, a fundamental result discovered by Church, Ellenberg, Farb, and Nagpal.
\end{abstract}

\maketitle

\section{Preliminaries}

The \emph{locally Noetherian property} of the category $\FI$, whose objects and morphisms are finite sets and injections between those sets respectively, was firstly established by Church, Ellenberg, and Farb in \cite{CEF} over fields of characteristic 0, and by them and Nagpal in \cite{CEFN} over arbitrary commutative Noetherian rings. Another proof of this fundamental result was obtained in \cite{SS} by Sam and Snowden using the Gr\"{o}bner methods translated from commutative algebra. In this note we provide two new proofs via a homological approach developed recently in a series of paper \cite{CE, G, Li, L2, LR, LY, R} by Church, Ellenberg, Gan, Ramos, Yu and the author, and extend some results of finitely generated $\FI$-modules over commutative Noetherian rings to those satisfying a certain weaker condition over arbitrary commutative rings.

As the reader can see, a few results in this note have been established in literature. The purpose to include them in this note is to give a brief self-contained homological approach. Since these results are either well known or as expected, we often omit proofs or only give sketches, and advice the reader to check references if necessary.

Throughout this note we let $G$ be a group, $\C$ be the skeletal full subcategory of $\FI_G$ whose objects are parameterized by nonnegative integers. Let $k$ be a commutative ring with identity, and $\tC$ be the $k$-linearization of $\C$. By $G_n$ we mean the endomorphism group of the object parameterized by $n$. We briefly call upon the following definitions. For details, please refer to \cite{LR}.

A \emph{$\C$-module $V$} is a covariant functor from $\C$ to $k \Mod$, the category of $k$-modules. We denote $V_n$ the value of $V$ on the object parameterized by $n$. Let $\mi$ be the ideal of $\tC$ spanned by all non-invertible morphisms in $\C$. The \emph{torsion degree} and \emph{$s$-th homological degree} of $V$ are defined respectively as:
\begin{align*}
\td(V) & = \sup \{ n \mid \Hom_{\tC} (kG_n, V) \neq 0 \}, \\
\hd_s(V) & = \td (\Tor ^{\tC}_s (\tC/\mi, V)), \quad s \geqslant 0
\end{align*}
where the torsion degree is set to be $-\infty$ by convention whenever the above set is empty, and in that case $V$ is said to be \emph{torsion free}. We call the $0$-th homological degree the \emph{generating degree} of $V$, and denote it by $\gd(V)$.

\begin{definition}
Let $V$ be a $\C$-module. It is \emph{locally finite} if for every $n \geqslant 0$, $V_n$ is a finitely generated $kG_n$-module. It is \emph{generated in finite degrees} if $\gd(V) < \infty$.
\end{definition}

\begin{remark}
Clearly, a $\C$-module $V$ is finitely generated if and only if it is locally finite and generated in finite degrees.
\end{remark}

\begin{definition}
Let $V$ be a $\C$-module. An element $v \in V$ is \emph{torsion} if one has $\mi^n v = 0$ for $n \gg 0$. The $\C$-module $V$ is \emph{torsion} if every element in it is a torsion element.
\end{definition}

Clearly, a $\C$-module $V$ is torsion free if and only if it has no torsion elements.

\begin{lemma}
Every finitely generated torsion module has finite torsion degree.
\end{lemma}

\begin{proof}
Let $V$ be a finitely generated torsion module. Choose a finite set of generators $\{v_1, \, \ldots, \, v_s\}$ and suppose that $v_i \in V_{N_i}$. Since these generators are torsion elements, there exist $n_1, \, \ldots, \, n_s$ such that $\mi^{n_i} v_i = 0$. Let $n = \max \{n_1, \, \ldots, \, n_s \}$. Then $\mi^n v_i = 0$ for $1 \leqslant i \leqslant s$. Since $V$ is generated by $v_1, \, \ldots, \, v_s$, one knows
\begin{equation*}
\td(V) \leqslant \max \{ n_i + N_i \mid 1 \leqslant i \leqslant s \} < \infty.
\end{equation*}
\end{proof}

However, if $V$ is only generated in finite degrees, then $\td(V)$ might be $\infty$. For example, let $P = \tC e_0$ where $e_0$ is the identity of $G_0$. Then the module
\begin{equation*}
V = \bigoplus _{i \geqslant 1} (P /\mi^iP)
\end{equation*}
is generated in degree 0, but its torsion degree is $\infty$.

For $V \in \C \Mod$, there is a short exact sequence
\begin{equation*}
0 \to V_T \to V \to V_F \to 0,
\end{equation*}
where
\begin{equation*}
V_T = \{ v \in V \mid \mi^n v = 0 \text{ for } n \gg 0 \}
\end{equation*}
is the \emph{torsion part} of $V$, and $V_F$ is the \emph{torsion free part} of $V$. Clearly, $\td(V) = \td(V_T)$.

Let $\Sigma$ and $D$ be the \emph{shift functor} and \emph{derivative functor} introduced in \cite{CEF} and \cite{CE} respectively. The following result has been described in \cite[Proposition 2.4]{LY} for finitely generated $\FI$-modules. However, we should point out that the result holds in a much more general framework. That is, it holds for any groups $G$ and any $\FI_G$-modules.

\begin{lemma}
Let $V$ be a $\C$-module. Then $\gd(DV) = \gd(V) -1$ whenever $DV \neq 0$.
\end{lemma}

\begin{proof}
The proof provided in \cite{LY} actually works for all $\C$-modules generated in finite degrees since applying the shift functor to a projective module generated in a certain degree $n$ we always get a direct sum of two projective modules generated in degree $n$ and degree $n-1$ respectively, so we only have to show that $\gd(DV) = \infty$ when $\gd(V) = \infty$. In that case, for an arbitrarily large $N$, one can construct a short exact sequence
\begin{equation*}
0 \to V' \to V \to V'' \to 0
\end{equation*}
such that $V'$ is the submodule generated by $\bigoplus _{i \leqslant N} V_i$. Because $\gd(V) = \infty$, $V''$ cannot be 0. Let $n$ be the minimal number such that $V''_n \neq 0$. By definition, $n > N$. Applying the right exact functor $D$ one gets a surjection $DV \to DV'' \to 0$. Note that $DV''_{n-1} \neq 0$ and $DV''_i = 0$ for $i < n-1$. Therefore, $\gd(DV'') \geqslant n-1$. This forces $\gd(DV) \geqslant n-1 \geqslant N$. Since $N$ is arbitrarily chosen, one knows that $\gd(DV) = \infty$ as well.
\end{proof}

\section{The first proof}

In this section we let $k$ be a commutative Noetherian ring, and let $G$ be a finite group. Note that $kG_n$ is a left Noetherian ring. Denote the category of locally finite $\C$-modules by $\C \lfMod$. The reader can easily see that it is an abelian category. Moreover, the shift functor $\Sigma$ preserves locally finite modules.

\begin{lemma}
A $\C$-module $V \in \C \lfMod$ is finitely generated if and only if $\gd(V) < \infty$. Moreover, if $V \in \C \lfMod$ is torsion, then $V$ is  finitely generated if and only if $\td(V) < \infty$.
\end{lemma}

\begin{proof}
The first statement follows from Remark 1.2. The only if direction of the second statement follows from Lemma 1.4, and the other direction of this statement is trivial since for a torsion module one always has $\gd(V) \leqslant \td(V)$.
\end{proof}

The following lemma is crucial for proving the main result.

\begin{lemma}
Let $k$ be a commutative Noetherian ring. If $V \in \C \lfMod$ is finitely generated, so is $V_T$. In particular, $\td(V) < \infty$.
\end{lemma}

\begin{proof}
Since $V \in \C \lfMod$, so is $V_T$ as $\C \lfMod$ is an abelian category. Moreover, by Remark 1.2, one knows that $\gd(V) < \infty$, and one has to show that $\gd(V_T) < \infty$. We use induction on $\gd(V)$ since this is supposed to be a finite number. Nothing needs to show for $\gd(V) = - \infty$, so we suppose $\gd(V) = n \geqslant 0$.

The short exact sequence
\begin{equation*}
0 \to V_T \to V \to V_F \to 0
\end{equation*}
induces the following commutative diagram
\begin{equation*}
\xymatrix{
0 \ar[r] & V_T \ar[r] \ar[d]^{\alpha} & V \ar[r] \ar[d]^{\beta} & V_F \ar[r] \ar[d]^{\delta} & 0\\
0 \ar[r] & \Sigma V_T \ar[r] \ar[d] & \Sigma V \ar[r] \ar[d] & \Sigma V_F \ar[r] \ar[d] & 0\\
0 \ar[r] & D(V_T) \ar[r] & DV \ar[r] & DV_F \ar[r] & 0
}
\end{equation*}
Note that $V_F$ is a torsion free module, so the map $\delta$ is injective according to \cite[Lemma 3.6]{CE}. Therefore, by the snake lemma, the bottom row is exact.

Clearly, all modules appearing in this diagram are locally finite. Moreover, by Lemma 1.5, $\gd(DV) < \gd(V)$. By the induction hypothesis, $(DV)_T$ is a finitely generated $\C$-module; or equivalently, $\td((DV)_T) < \infty$ by the previous lemma. Of course, in general one cannot expect $D(V_T) \cong (DV)_T$ since $DV_F$ might not be torsion free. However, since $\Sigma V_T$ is a torsion module (this can be checked easily using the definitions of torsion modules and the shift functor), its quotient $D(V_T)$ is torsion as well. Therefore, $D(V_T)$ is isomorphic to a submodule of the torsion part $(DV)_T$, so
\begin{equation*}
\td(D(V_T)) \leqslant \td((DV)_T) < \infty.
\end{equation*}
Consequently, $\gd(DV_T) < \infty$. By Lemma 1.5, $\gd(V_T)$ must be finite as well. That is, $V_T$ is finitely generated, and $\td(V) = \td(V_T) < \infty$. The conclusion then follows from induction.
\end{proof}

Now we are ready to give a new proof for the following well known result.

\begin{theorem}
The category $\FI_G$ is locally Noetherian over any commutative Noetherian ring $k$.
\end{theorem}

\begin{proof}[The first proof]
By a standard homological argument, we only need to show the following statement: for $n \geqslant 0$, every submodule $V$ of the projective module $\tC e_n$ is finitely generated, where $e_n$ is the identity of the endomorphism group $G_n$. Since $\tC e_n$ is obviously a locally finite $\C$-module, by Lemma 2.1, it suffices to show $\gd(V) < \infty$.

We use induction on $n$. If $\tC e_n = 0$; that is, $n < 0$, nothing needs to show. Otherwise, the short exact sequence
\begin{equation*}
0 \to V \to \tC e_n \to W \to 0
\end{equation*}
induces the following commutative diagram
\begin{equation*}
\xymatrix{
0 \ar[r] & V \ar[r] \ar[d]^{\alpha} & \Sigma V \ar[r] \ar[d]^{\beta} & DV \ar[r] \ar[d]^{\delta} & 0\\
0 \ar[r] & \tC e_n \ar[r] \ar[d] & \Sigma (\tC e_n) \ar[r] \ar[d] & D (\tC e_n) \cong (\tC e_{n-1})^{n|G|} \ar[r] \ar[d] & 0\\
0 \ar[r] & W \ar[r] & \Sigma W \ar[r] & DW \ar[r] & 0.
}
\end{equation*}
As explained in the proof of \cite[Theorem 2.4]{L2}, the maps $\alpha$ and $\beta$ are injective, and the third column gives rises to two short exact sequences
\begin{align*}
& 0 \to K \cong \ker \delta \to DV \to U \to 0,\\
& 0 \to U \to (\tC e_{n-1})^{n|G|} \to DW \to 0
\end{align*}
with $\gd(K) = \td(K) = \td(W)$.

The induction hypothesis and the second short exact sequence tell us that $\gd(U) < \infty$. Since $\gd(W) < \infty$, the previous lemma tells us that $\gd(K) < \infty$. Therefore, from the first exact sequence we conclude that $\gd(DV) < \infty$. By Lemma 1.5, $\gd(V) < \infty$. The conclusion then follows from induction.
\end{proof}

\section{Modules presented in finite degrees}

In this section $G$ is an arbitrary group and $k$ is a commutative ring. Although modules with finite generating degrees naturally extend finitely generated modules, their homological behaviors are not good enough since submodules need not be generated in finite degrees again. Instead, we consider $\C$-modules $V$ \emph{presented in finite degrees}; that is, both $\gd(V)$ and $\hd_1(V)$ are finite. This happens if and only if there is a presentation
\begin{equation*}
P^1 \to P^0 \to V \to 0
\end{equation*}
such that both $\gd(P^0)$ and $\gd(P^1)$ are finite, as described in \cite{CE, R}.

We recall the following important result:

\begin{theorem}[\cite{CE, L2, R}]
Let $V$ be a $\C$-module over a commutative ring $k$. Then for $s \geqslant 1$,
\begin{align*}
\td(V) & \leqslant \gd(V) + \hd_1(V) - 1,\\
\hd_s(V) & \leqslant \gd(V) + \hd_1(V) + s - 1.
\end{align*}
In particular, $\td(V)$ is finite.
\end{theorem}

\begin{remark} \normalfont
These two inequalities were proved by Church and Ellenberg in \cite{CE} for $\FI$-modules over any commutative ring, and the second inequality was generalized to $\FI_G$-modules by Ramos in \cite{R}. Another proof was given in \cite{L2} by the author, which is still valid in this general setup.
\end{remark}

We get the following corollary.

\begin{corollary}
Let $0 \to U \to V \to W \to 0$ be a short exact sequence of $\C$-modules, and suppose that $V$ is presented in finite degrees and $U$ is generated in finite degrees. Then all terms are presented in finite degrees.
\end{corollary}

\begin{proof}
Note that since $V$ is presented in finite degrees, all of its homological degrees are finite. Applying the homology functor to the above short exact sequence and considering the induced long exact sequence, we deduce that $W$ is presented in finite degrees, and hence all of its homological degrees are finite. Therefore, all homological degrees of $U$ are finite.
\end{proof}

The Noetherian property of $\C$ over a commutative Noetherian ring is equivalent to saying that the category of finitely generated $\C$-modules is abelian. We prove a similar result for $\C$-modules presented in finite degrees, which is a straightforward consequence of the finiteness of \emph{Castelnuovo-Mumford regularity} (see \cite{CE} for a definition). It was also proved by Ramos in \cite{R1} independently at almost the same time.

\begin{proposition}
The category of $\C$-modules presented in finite degrees is abelian.
\end{proposition}

\begin{proof}
We only need to show that this category is closed under kernels and cokernels of morphisms. Let $\alpha: U \to V$ be a homomorphism between two $\C$-modules presented in finite degrees. It gives rise to two short exact sequences
\begin{align*}
& 0 \to \kernel \alpha \to U \to \Image \alpha \to 0,\\
& 0 \to \Image \alpha \to V \to \coker \alpha \to 0.
\end{align*}
As a quotient of $U$, $\Image \alpha$ is clearly generated in finite degrees. By the previous corollary, all terms in the second short exact sequence are presented in finite degrees. Since all homological degrees of $U$ and $\Image \alpha$ are finite, so are homological degrees of $\kernel \alpha$ by the first short exact sequence.
\end{proof}

In a previous version of this note, the author asked the following question: If $V$ is a $\C$-module satisfying $\td(V) < \infty$ and $\gd(V) < \infty$, is $\td(DV) < \infty$ as well? This question was answered by Ramos in \cite{R1}. His following result, which is a direct corollary of \cite[Theorem D]{CE}, implies a confirmative answer to the above question.

\begin{lemma} [Proposition 3.3, \cite{R1}]
Let $V$ be a submodule of a torsion free $\C$-module $W$. Then $\td(W/V) < \infty$.
\end{lemma}

\begin{proof}
The injection $V \to W$ gives the following commutative diagram
\begin{equation*}
\xymatrix{
\ldots \ar[r] & V_n \ar[r] \ar[d] & V_{n+1} \ar[r] \ar[d] & \ldots\\
\ldots \ar[r] & W_n \ar[r] \ar[d]^{p_n} & W_{n+1} \ar[r] \ar[d]^{p_{n+1}} & \ldots\\
\ldots \ar[r] & W_n/V_n \ar[r] & W_{n+1}/V_{n+1} \ar[r] & \ldots
}
\end{equation*}
where $p_n$ and $p_{n+1}$ are projections.

If $\td(W/V) = \infty$, then for a large enough $n$, one can find an element $0 \neq \bar{w}_n \in W_n/V_n$ such that its image in $W_{n+1}/V_{n+1}$ is 0. Let $w_n \in W_n$ be an element in $p_n^{-1} (\bar{w}_n)$. Then the image of $w_n$ in $W_{n+1}$ is also contained in $V_{n+1}$. However, according to \cite[Theorem D]{CE}, \footnote{As explained in \cite{R}, the conclusion actually holds for $\FI_G$-modules, though only $\FI$-modules were considered in this theorem.} for $n \gg 0$, by taking $a = 1$, one has $V_{n+1} \cap W_n = V_n$. That is, $w_n \in V_n$, so $\bar{w}_n = 0$. The conclusion follows by this contradiction.
\end{proof}

\begin{corollary}
Let $V$ be a $\C$-module satisfying $\td(V) < \infty$ and $\gd(V) < \infty$. Then $\td(DV) < \infty$.
\end{corollary}

\begin{proof}
The exact sequence $0 \to V_T \to V \to V_F \to 0$, where $V_T$ is the maximal torsion submodule of $V$ and $V_F$ is torsion free, induces the following diagram
\begin{equation*}
\xymatrix{
0 \ar[r] & V_T \ar[r] \ar[d] & V \ar[r] \ar[d] & V_F \ar[r] \ar[d] & 0\\
0 \ar[r] & \Sigma V_T \ar[r] \ar[d] & \Sigma V \ar[r] \ar[d] & \Sigma V_F \ar[r] \ar[d] & 0\\
0 \ar[r] & DV_T \ar[r] & DV \ar[r] & DV_F \ar[r] & 0.
}
\end{equation*}
Indeed, since the map $V_F \to \Sigma V_F$ is injective, the bottom row is exact by the snake lemma. Therefore, to validate the conclusion, it suffices to show $\td(DV_F) < \infty$ and $\td(DV_T) < \infty$. But the first inequality follows from the previous lemma, and the second one is clear since $V_T$ is a torsion module with $\td(V_T) = \td(V) < \infty$.
\end{proof}

In \cite{R1}, Ramos proved the following beautiful result. For the convenience of the reader, we include a slightly different proof.

\begin{theorem}[Theorem B, \cite{R1}]
A $\C$-module $V$ is presented in finite degrees if and only if both $\gd(V)$ and $\td(V)$ are finite.
\end{theorem}

\begin{proof}
If $V$ is presented in finite degrees, then $\td(V) < \infty$ by Theorem 3.1. Conversely, suppose that $\gd(V) < \infty$ and $\td(V) < \infty$, and we want to show that $\hd_1(V) < \infty$. The exact sequence $0 \to V_T \to V \to V_F \to 0$ tells us that it suffices to prove $\hd_1(V_T) < \infty$ and $\hd_1(V_F) < \infty$. The first inequality follows from \cite[Theorem 1.5]{Li}. To prove the second one, we note that $\gd(DV_F)$ and $\td(DV_F)$ are finite as well. Moreover, $\gd(DV_F) \leqslant \gd(DV) < \gd(V)$. By an induction on the generating degree, we deduce that $\hd_1(DV_F) < \infty$, so $\hd_1(V_F) < \infty$ as wanted by \cite[Corollary 2.10, Proposition 2.11]{LY}. \footnote{The assumption that $V$ is finitely generated in these two results is inessential since the arguments apply to all modules.}
\end{proof}

This theorem can be used to establish the following important result. For a definition of \emph{filtered modules}, see \cite{L2, LR, LY}.

\begin{theorem}
Let $k$ be a commutative ring, and let $V$ be a $\C$-module presented in finite degrees. Then $\Sigma_n V$ is a filtered module for $n \gg 0$.
\end{theorem}

\begin{proof}
To prove this result, we use the following key proposition: Let $V$ be a torsion free $\C$-module with $\gd(V) < \infty$. If $DV$ is filtered, so is $V$. This was proved in \cite[Section 3]{LY} for finitely generated $\C$-modules over Noetherian rings, but the argument actually works for modules generated in finite degrees over any commutative ring. Another proof of this fact was recently described in \cite{G}.

Now let $V$ be an $\C$-module presented in finite degrees. By the above theorem and Corollary 3.6, both $\td(V)$ and $\td(DV)$ are finite. Therefore, for $N \gg 0$, one gets a short exact sequence
\begin{equation*}
0 \to \Sigma_N V \to \Sigma_{N+1} V \to \Sigma_N D V \to 0,
\end{equation*}
such that all terms are torsion free. By the induction hypothesis on generating degree, $D \Sigma_N V \cong \Sigma_N DV$ is filtered. By the key proposition, $\Sigma_N V$ is filtered.
\end{proof}

Theorem 3.7 and Lemma 2.2 also provide us an almost trivial proof of the Noetherianity of $\FI_G$ over commutative Noetherian rings, where $G$ is a finite group.

\begin{proof}[The second proof of local Noetherianity of $\FI_G$]
Take an arbitrary $n \geqslant 0$ and an arbitrary submodule $V$ of $P = \tC e_n$. One needs to show that $\gd(V) < \infty$. Consider the short exact sequence
\begin{equation*}
0 \to V \to P \to W \to 0.
\end{equation*}
Clearly, $W$ is finitely generated, so $\gd(W) < \infty$. By Lemma 2.2, $\td(W) < \infty$. By Theorem 3.7, $W$ is presented in finite degrees, so $\hd_1(W) < \infty$. By \cite[Corollary 2.10]{LY}, $\gd(V) < \infty$.
\end{proof}

In the rest of this note we investigate \emph{local cohomology groups} $H^i_\mi(V)$ (see \cite{LR} for their definition) of $\C$-modules presented in finite degrees. Since arguments in \cite{LR} are still valid in this wider framework, we omit the proofs. These results were proved by Ramos independently in \cite{R1}, where details are included. 

\begin{theorem}[\cite{LR}, Theorem E]
Let $k$ be a commutative ring, and let $V$ be a $\C$-module presented in finite degrees. Then there is a finite complex of filtered modules (except the first term $V$)
\begin{equation*}
\mathcal{C^\bullet}V : 0 \rightarrow V \rightarrow F^0 \rightarrow \ldots \rightarrow F^{n} \rightarrow 0
\end{equation*}
such that the cohomologies $H^i(\mathcal{C}^\bullet V)$ are torsion modules satisfying
\begin{equation*}
\begin{cases} \td(H^{i}(C^{\bullet}V)) = \td(V) &\text{ if $i = -1$}\\ \td(H^{i}(C^{\bullet}V)) \leqslant 2\gd(V) - 2i  - 2 &\text{ if $0 \leq i \leq n$.}\end{cases}
\end{equation*}
Moreover, $H^{i+1}_{\mi}(V) \cong H^i(\mathcal{C}^\bullet V)$ for $i \geqslant -1$.
\end{theorem}

\begin{theorem}[\cite{LR}, Theorem F]
Let $V$ be as in the previous theorem. Then
\begin{enumerate}
\item the smallest $N$ such that $\Sigma_N V$ is a filtered module is
\begin{equation*}
\max\{\td(H^i_\mi(V))\mid i \geqslant 0\} + 1 \leqslant \max\{\td(V),2\gd(V) - 2\} + 1.
\end{equation*}
\item The Castelnuovo-Mumford regularity of $V$, $\reg(V)$ satisfies the bounds
\begin{equation*}
\reg(V) \leqslant \max\{\td(H^i_\mi(V)) + i\} \leqslant \max\{2\gd(V ) - 1, \td(V )\}.
\end{equation*}
\end{enumerate}
\end{theorem}

\end{document}